\documentclass{article}
\usepackage{graphicx}
\graphicspath{{images/}}
\usepackage{wrapfig}
\usepackage{amsthm}
\newtheorem{theorem}{Theorem}[section]
\newtheorem{lemma}[theorem]{Lemma}

\newtheorem{corollary}[theorem]{Corollary}
\newtheorem{remark}[theorem]{Remark}
\newtheorem{defn}[theorem]{Definition}

\usepackage{mathrsfs}
\usepackage{dsfont}
\usepackage{amsmath}
\usepackage[margin=1.0in]{geometry}

\title{The level-set flow of the topologist's sine curve is smooth}

\author{Casey Lam\thanks{The first author was supported by the Marianna Polonsky Slocum Memorial Fund.}\and Joseph Lauer}

\date{January 11, 2016}

\begin{document}

\maketitle

\begin{abstract}
In this note we prove that the level-set flow of the topologist's sine curve is a smooth closed curve.  In~\cite{Lauer1} it was shown by the second author that under level-set flow, a locally-connected set in the plane evolves to be smooth, either as a curve or as a positive area region bounded by smooth curves.  Here we give the first example of a domain whose boundary is not locally-connected for which the level-set flow is instantaneously smooth.  Our methods also produce an example of a non path-connected set that instantly evolves into a smooth closed curve.    \end{abstract}


\section{Introduction}

The mean curvature flow of a compact hypersurface necessarily develops a singularity in finite time.  At that point the classical solution ceases to exist and several so-called weak solutions have been defined that allow the evolution to be extended.  One such example is level-set flow (see Section 2 for a definition), which was first investigated independently by Chen-Giga-Goto~\cite{CGG} and Evans-Spruck~\cite{ES}.  Level-set flow exists for \emph{any} compact initial data and agrees with smooth mean curvature flow whenever the latter is defined.

While the main use of level-set flow is certainly to analyze and flow through singularities it is nevertheless interesting to consider nonsmooth initial data which do no occur in this context.  In this paper we study the level-set flow of the standard topologist's sine curve, which we denote by $T_t$.  For any given $t>0$ elementary geometric arguments can be used to show that $T_t$ has Lebesgue measure zero and that the complement of $T_t$ has two connected components.  However, these arguments cannot rule out the possibility that the characteristic topological properties of the topologist's sine curve are preserved.  The main result of this paper is the following:

\begin{theorem}\label{main}
$T_t$ is a unique smooth curve shortening flow for $t>0$.
\end{theorem}

In fact what we show here is that $T_t$ instantly becomes locally-connected.  At that point we make use of the classification proved by the second author in~\cite{Lauer1}.  One restatement of that result is the following:

\begin{theorem}\label{Lauer}(\cite{Lauer1}, Theorem 1.1) Let $K \subset \mathds{R}^2$ be locally-connected, connected and compact.  Then for each~$t>0$, $\partial K_t$ is a (perhaps empty) finite union of smooth closed curves.
\end{theorem}

In Theorem~\ref{Lauer} there are essentially three possibilities: The first is that $K_t$ instantly evolves as a smooth closed curve, the second is that $K_t$ is a positive area set with smooth boundary and the third is that $K_t$ vanishes instantly.

The question of nonsmooth initial data for mean curvature flow has been studied extensively: Ecker and Huisken~\cite{EH1} first proved that Lipschitz entire graphs have a smooth evolution and later~\cite{EH2} showed that the same is true for curves that satisfy a uniform local Lipschitz condition.  In her thesis, Clutterbuck~\cite{Cl} proved the existence of a smooth evolution for bounded graphs which were merely continuous. In~\cite{Lauer1} the case of planar curves was studied.  There the first examples were given of sets with Hausdorff dimension $>$ 1 (the Koch snowflake for example) that evolve to be smooth under level-set flow.  More recently Hershkovits~\cite{Hersh} has given the first examples of this behaviour in arbitrary dimensions by studying $(\epsilon, R)$-Reifenberg sets for~$\epsilon$~sufficiently small.


\subsection{Overview}

In this section we give an outline of the proof of Theorem~\ref{main}.  As usual when studying nonsmooth initial data for a parabolic PDE the proof proceeds by establishing uniform estimates on a sequence of smooth approximations.

We begin by defining the initial data:

\begin{defn} [Topologist's sine curve] A topologist's sine curve $T$ is the union of the graph of $y=\sin(1/x)$ for $x\in (0,\beta)$, the line segment $V=\{0\}\times[-1,1]$, and a smooth arc (disjoint from the first two pieces except at its endpoints) connecting these two pieces.  See Figure 1.
\end{defn}

\begin{figure}\label{TSCpic}
    \centering
    \includegraphics[scale=0.18]{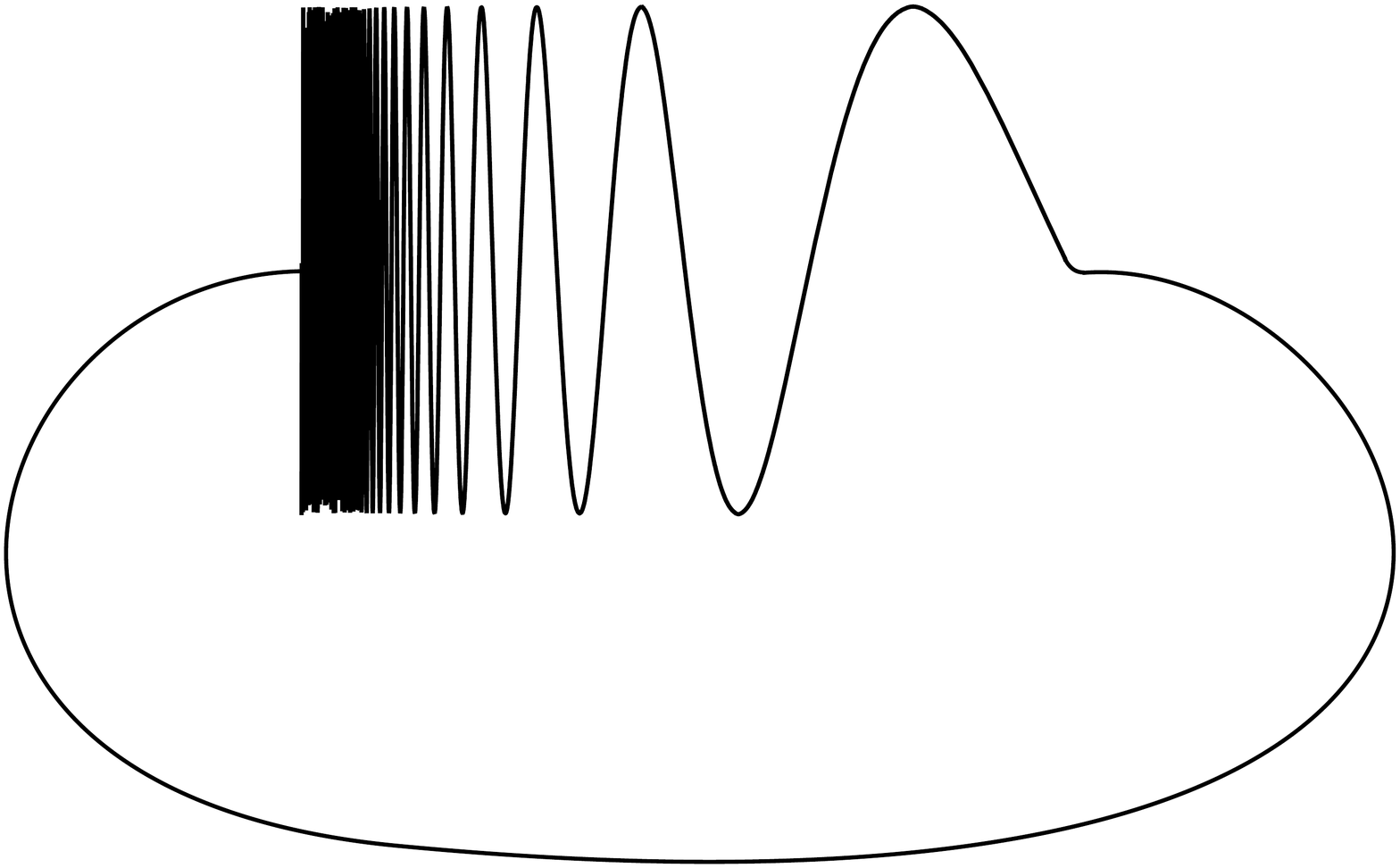}
    \caption{The topologist's sine curve.}
\end{figure}

The exact definition, i.e. how much of the graph $y=\sin(1/x)$ is taken and which arc is used to connect the end of it back to $V$, is not important.  Indeed, it will be clear from the argument that Theorem~\ref{main} is valid for any diffeomorphic image of $T$ as long as the compactifying segment remains linear. 

In Section~\ref{sec_lsfdef} we recall that the level-set flow of $T$ is determined by two sequences of smooth approximations, one contained in each of the two components of $\mathds{R}^2\setminus T$.  Let $\gamma_n$ be one such sequence. Since~$\gamma_n$ is chosen so as to Hausdorff converge to $T$ it follows that the lengths, denoted $\mathcal{L}(\gamma_n)$, are unbounded.  Indeed, this is even true locally for $x\in V$.  In that case for any $\epsilon>0$ 
$$
\sup_{n\in\mathds{N}} \{\mathcal{L}(\gamma_n \cap B_\epsilon (x))\}=\infty,
$$
where $B_\epsilon(x)=\{y\mid |x-y|<\epsilon\}$.  The main step in the proof of Theorem~\ref{main} is showing that the supremum above is finite if one evolves the approximations simultaneously for a short time.  Namely, we prove the following: 

\begin{theorem}\label{TSC}
For each $t > 0$ and $x \in \mathds{R}^2$, there exists $\epsilon = \epsilon(x, t)>0$ such that
$$
\sup_{n\in\mathds{N}}\mathcal{L}((\gamma_n)_t \cap B_\epsilon (x))<\infty.
$$
\end{theorem}

And since the $\epsilon$ in Theorem~\ref{TSC} does not depend on $n$ the following Corollary is immediate:

\begin{corollary}\label{boundapprox} For each $t>0$ 
$$
\sup_{n\in\mathds{N}}\mathcal{L}((\gamma_n)_t)<\infty.
$$
\end{corollary}

In Section 5 we show that Corollary~\ref{boundapprox} implies that $\mathcal{H}^1(\partial T_t)<\infty$ for $t>0$, where $\mathcal{H}^1$ denotes the Hausdorff 1-measure, and hence that $\partial T_t$ is locally-connected (See for example~\cite{F}).  By Theorem~\ref{Lauer} this implies that $\partial T_t$ is smooth.  At this point we have essentially proved the existence portion of Theorem~\ref{main} but have made no claim of uniqueness as it is \emph{a priori} possible the boundary components defined by the two sequences of approximations are distinct.  But Theorem~\ref{main} then follows from the observation, see Lemma~\ref{measure}, that $T_t$ has Lebesgue measure zero. 


\subsection{The proof of Theorem~\ref{TSC}} The proof of Theorem~\ref{TSC} requires two separate arguments.  For $x\notin V$ the conclusion of Theorem~\ref{TSC} holds at $t=0$ and a simple argument counting intersections with a grid of static transverse lines shows that the uniform finiteness persists. 

For $x\in V$ the argument is more involved.  In this case we instead count intersections with a translating solution to curve shortening flow, the so-called grim reaper.   We choose the initial data $u$ such that $u\cap V=\emptyset$ and yet $u_t$ passes through~$V$ before time~$t_0$. This can be done since by parabolic scaling one can construct a grim reaper which is arbitrarily thin, and hence arbitrarily fast. See Figure 2.  

In Lemma~\ref{intersect} we show that $|u\cap T|$, and hence $|u\cap\gamma_n|$ are bounded and use the fact that the number of intersections is non-increasing under curve shortening flow to conclude that $|(\gamma_n)_{t_0}\cap u_{t_0}|$ is also uniformly bounded.  We then choose $\epsilon>0$ small enough so that the subset of $u$ that passes through $B_\epsilon(x)$ is nearly horizontal.  At $t_0$ we have that $(\gamma_n)_{t_0}\cap B_\epsilon(x)$ has a bounded number of intersections with each leaf in two transverse foliations of $B_\epsilon(x)$.  The vertical foliation is exactly linear while the horizontal one is $\mathcal{C}^1$-close to linear.  The result then follows since the bi-Lipschitz constant needed to map the pair of foliations to the standard grid is well-controlled. 

\begin{figure}\label{grimreaper}
    \centering
    \includegraphics[scale=0.2]{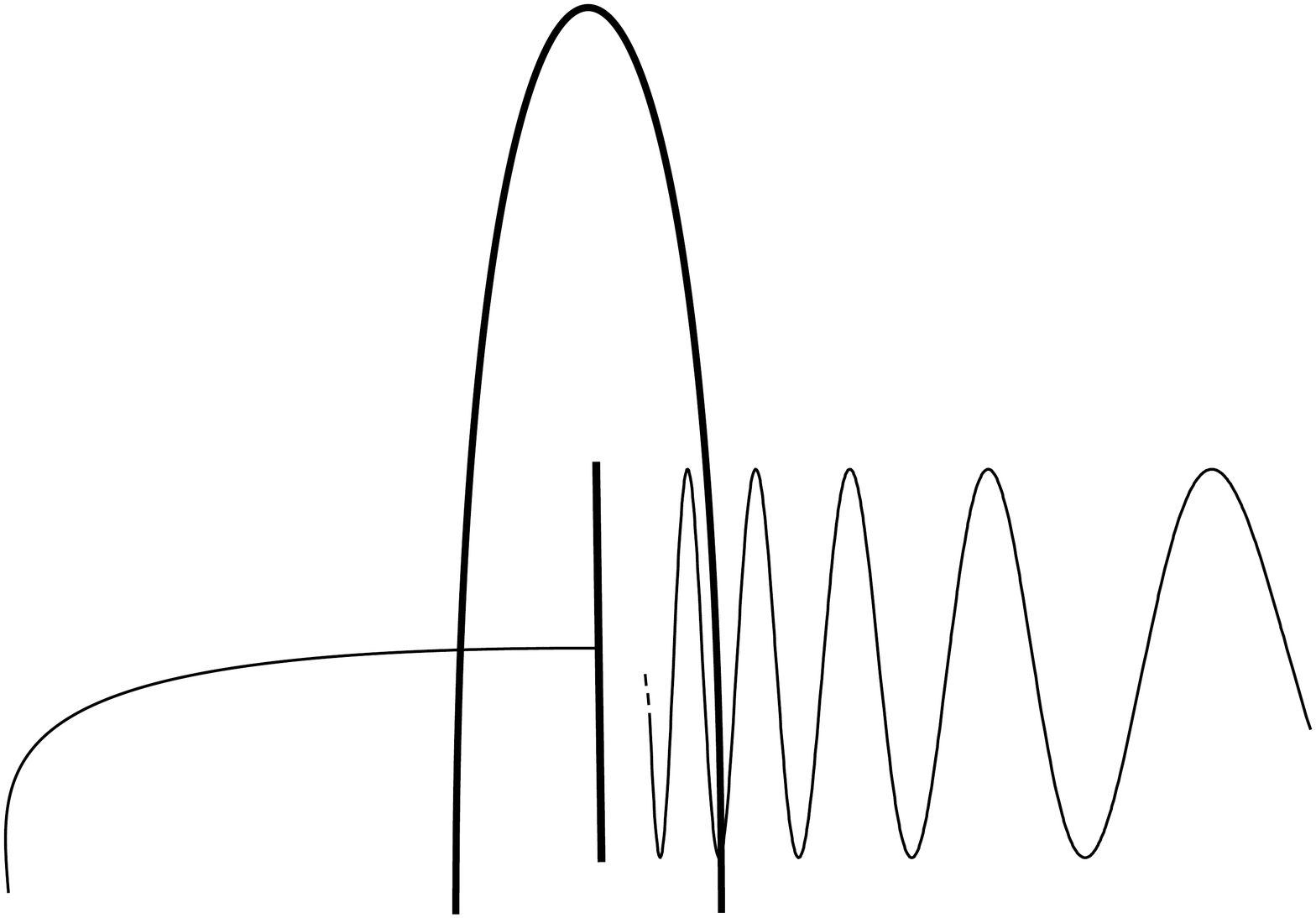}
    \caption{The grim reaper initially contains $V$ in its convex hull. The width, which is inversely proportional to its speed, is chosen so that the evolution passes through $V$ before $t=t_0$.}
\end{figure}


\subsection{Organization of Paper}
In Section~\ref{sec_lsfdef} we review the definition of level-set flow and specify the properties we require in an approximating sequence. In Section~\ref{sec_length} we compute a simple length estimate for transformations that will be used twice in Section~\ref{sec_locallyfinite}, where we prove Theorem~\ref{TSC}. In Section~\ref{sec_locallyconnected} we verify that Theorem~\ref{TSC} implies that the boundary components of  $T_t$ have finite $\mathcal{H}^1$-measure when $t>0$, which in turn implies that they are locally-connected.  In Section~\ref{sec_proof} we complete the proof of Theorem~\ref{main}.


\section{Level set flow and approximating sequences}\label{sec_lsfdef}

Mean curvature flow was first studied by Brakke~\cite{Br} in the context of geometric measure theory and later Huisken~\cite{H} developed the following classical formulation for smooth hypersurfaces:

\begin{defn} [Mean curvature flow] Let $\{M_t\}_{t\in(a,b)}$ be an evolving family of hypersurfaces in $\mathds{R}^{n+1}$. We say that $M_t$ is moving by mean curvature flow if it satisfies the nonlinear parabolic equation
$$ \frac{\partial}{\partial{t}}x = \vec{H}(x), \hspace{0.5cm}
x \in M_t , \hspace{0.5cm}
t \in (a,b),$$
where $\vec{H}(x)$ is the well-defined mean curvature vector at the point $x$.  
\end{defn}

As discussed in the Introduction the existence of finite-time singularities has led to several definitions of a weak solution, one of which level-set flow.  In this section we review some basic facts about level-set flow and the smooth approximations that determine the flow in the specific case of $\mathds{R}^2$. The geometric formulation in terms of weak-set flows given below is due to Ilmanen~\cite{I}.  For further background see~\cite{CGG}~\cite{ES}.

\begin{defn}\label{weaksetflow}\emph{(Weak-set flow)}
Let $K \subset \mathds{R}^{n+1}$ be compact, and let ${\{K_t\}}_{t \geq 0}$ be a 1-parameter family of compact sets with $K_0 = K$, such that the space-time track $\cup (K_t \times \{t\}) \subset \mathds{R}^2$ is closed. Then  ${\{K_t\}}_{t \geq 0}$ is a \emph{weak-set flow} for $K$ if for every smooth mean curvature flow $\Sigma_t$ defined on $[a,b] \subset [0, \infty]$ we have 
$$
K_a \cap \Sigma_a = \emptyset \rightarrow K_t \cap \Sigma_t = \emptyset
$$
for each $t \in [a,b]$.
\end{defn}

We are interested in one particular weak set flow: 

\begin{defn}\emph{(Level-set flow)
The \emph{level-set flow} of a compact set $K \subset \mathds{R}^{n+1}$ is the maximal weak-set flow. That is, a weak set flow $K_t$ such that if ${\hat{K}}_t$ is any other weak set flow, then ${\hat{K}}_t \subset K_t$ for all $t \geq 0$.}
\end{defn}

The existence of the level-set flow is established by taking the closure of the union of all weak set flows.  If the initial data is a smooth hypersurface then level-set flow agrees with mean curvature flow up to the first singular time.  On the other hand there is behaviour very different from smooth flows, the classic example being the figure eight which instantly \emph{fattens} into a smooth region with three boundary components.  Fattening in this context should be interpreted as non-uniqueness.

In $\mathds{R}^2$ the uniform existence of solutions to curve shortening flow, i.e. that the extinction time is proportional to the area bounded by the initial data, allows for a simple explicit definition.  The reason we do this here is that Lemma~\ref{measure} immediately implies the uniqueness portion of Theorem~\ref{main}, which is much more involved in higher dimensions, see~\cite{Hersh} for example. 

Let $\Omega\subset\mathds{R}^2$ be a bounded domain with topological boundary $K=\partial\Omega$.  Let $\alpha_n$, $\beta_n$ be sequences of smooth closed curves contained in $\Omega$ and $\mathds{R}^2\setminus\overline{\Omega}$ respectively such that 
$$
\lim_{n\to\infty}d_H(\alpha_n,K)\to 0,\:\:\:\:\:\lim_{n\to\infty}d_H(\beta_n,K)\to 0,
$$
where $d_H$ represents the Hausdorff distance on compact sets.  Then for each $t>0$ and $n\in\mathds{N}$ define $(A_n)_t$ to be the smooth annulus with boundary $\{(\alpha_n)_t,(\beta_n)_t\}$, where $(\alpha_n)_t$ and $(\beta_n)_t$ denote the evolution of $\alpha_n$ and $\beta_n$ by cuve shortening flow.  We then have the following:

\begin{lemma}\label{LSF} For all $t>0$ $$K_t=\bigcap_n (A_n)_t.$$
\end{lemma}

\begin{proof} Since $K\subset \mathrm{int}((A_n)_0)$ the definition of weak-set flow implies $K_t\subset \mathrm{int}(A_n)_t$ for each $t>0$ and hence that $K_t\subset \bigcap (A_n)_t$.  On the other hand it is easily checked that $\cap_n (A_n)_t$ is itself a weak-set flow, which provides the opposite containment.
\end{proof}

\begin{remark}  The same construction can be used to determine the level-set flow of an arbitrary compact set.  In that case there is one approximating sequence in each component of $\mathds{R}^2\setminus K$.  In order to generalize Lemma~\ref{LSF} to higher dimensions one must first produce curvature bounds that imply that the evolution of such sequences exist on a uniform time interval. 
\end{remark}

Using Lemma~\ref{LSF} it is easy to prove that the evolution of the topologist's sine curve has Lebesguse measure zero, which in particular implies that $T_t$ does not fatten.  Let $m(\cdot)$ denote the two-dimensional Lebesgue measure.

\begin{lemma} \label{measure} $m(T_t)=0$ for all $t>0$.
\end{lemma}

\begin{proof} Since $m(T_0)=0$ it follows that $\lim_{n\to\infty} m(A_n)=0$.  Under curve shortening flow the area bounded by a smooth curve decreases at a constant rate~\cite{GH} and hence the area contained in an evolving annulus is constant.  Thus $\lim_{n\to\infty} m((A_n)_t)=0$ for each $t>0$ and the result follows by Lemma~\ref{LSF}.
\end{proof}

In general the sequences $(\alpha_n)_t$ and $(\beta_n)_t$ may have distinct Hausdorff limits.  Indeed this happens in the case of a positive area Jordan curve which instantly evolves into a smooth annulus~\cite{Lauer1}.  In the case of the topologist's sine curve we show that each sequence limits onto a smooth curve and so Lemma~\ref{measure} implies that the two limits coincide.  


\subsection{Approximations} For the remainder of the paper we work with a sequence satisfying the following properties: 

\begin{defn}\label{goodapprox} [Allowable approximations]\label{goodapprox} For the Topologist's sine curve $T=\partial\Omega$ we define a sequence of smooth closed curves $\gamma_n$ with the following properties:
\begin{enumerate}
\item $\gamma_n\subset\Omega$,
\item the regions bounded by $\gamma_n$ form an increasing sequence of sets,
\item $d_H(\gamma_n,T)\to 0$ as $n\to\infty$, 
\item for any $\epsilon>0$, $\gamma_n$ converges smoothly to $T\setminus B_\epsilon(V)$, and
\item there exists $\omega\in\mathds{Z}^+$ such that each $\gamma_n$ is an $\omega$-graph.
\end{enumerate}
\end{defn}

\begin{remark}\label{outside} As discussed in the paragraph preceding Lemma~\ref{LSF} it is also necessary to consider a similarly defined sequence in $\mathds{R}^2\setminus\overline{\Omega}$.  It is clear from the proof that Theorem~\ref{TSC} is valid for either type of approximating sequence. 
\end{remark}

The final property above asserts that $\gamma_n$ intersects each vertical line at most $\omega$ times.  For the standard picture in Figure 1 each $\gamma_n$ can be chosen to be a double graph, but in general $\omega$ will depend on the arc added to connect the origin to the $\sin(1/x)$ portion of $T$.  

The way in which we use the fourth property above is the following: 

\begin{lemma}\label{intbound} Let $\ell$ be a line in $\mathds{R}^2\setminus V$.  Then 
$$
|\gamma_n\cap\ell|\leq 2|T\cap\ell|
$$
for $n$ sufficiently large.
\end{lemma}

The two appears on the right-hand side due to the fact that $\gamma_n$ will eventually intersect~$\ell$ either 0,1 or 2 times near a tangential intersection in $(T\setminus V)\cap\ell$. 

\begin{remark} In general all conclusions regarding the sequence $\gamma_n$ hold only for sufficiently large $n$.  Throughout the rest of the paper we assume that the sequence has been appropriately modified and ignore this point.
\end{remark}


\section{The length of a curve with bounded intersections}\label{sec_length}

In this section we establish the following elementary estimate which is used in two separate instances in Section~\ref{sec_locallyfinite}.  Let $\mathcal{B}_\epsilon=\{x\mid|x|<\epsilon\}\subset\mathds{R}^2$.      

\begin{theorem}\label{jacobian}
 Given $\epsilon,M>0$ and $d>1$ suppose that $f:\mathcal{B}_{\epsilon}\to\mathds{R}^2$ fixes the origin, is invertible, differentiable and that 
 $$
 d^{-1}<\|\mathrm{Jac}(f)\|_2< d.
 $$  
 If $\gamma$ is the (perhaps disjoint) restriction of a smooth curve to $\mathcal{B}_\epsilon$ and $|\ell\cap f(\gamma)|\leq M$ for each horizontal or vertical line $\ell$, then 
$$
\mathcal{L}(\gamma)\leq 4Md^2\epsilon
$$
\end{theorem}

\begin{remark}
Here $\|\mathrm{Jac}(f)\|_2$ denotes the (Euclidean) operator norm of the Jacobian of $f$.
\end{remark}

Theorem~\ref{jacobian} follows easily from the following Lemma which gives a bound on the length of a curve based on the maximum number of times it intersects a horizontal or vertical line: 

\begin{lemma}\label{straight}
Let $\epsilon>0$ and $\gamma$ be a differentiable curve in $\mathds{R}^2$ that intersects each vertical and horizontal line in $B=(-\epsilon,\epsilon)^2$ at most $M$ times. Then $$\mathcal{L}(\gamma \cap B) \leq 4M\epsilon.$$
\end{lemma}

\begin{proof}[Proof of Theorem~\ref{jacobian}]
Observe that $\|\mathrm{Jac}(f)\|_2 < d$ and the fact that $f(0)=0$ imply that
$$
f(\mathcal{B}_\epsilon)\subset\mathcal{B}_{d\epsilon}.
$$ 
Thus Lemma~\ref{straight} applied to $f(\gamma)$ yields
$$
\mathcal{L}(f(\gamma))<4Md\epsilon,
$$
and the result follows by taking the inverse.
\end{proof}

The proof of Lemma~\ref{straight}, which we include for completeness, goes as follows: $1)$ decompose the given curve into a minimal number of monotone segments, $2)$ bound the length of each such segment, and $3)$ bound the sum of the lengths of the their projections.  The third step is carried out in Lemma~\ref{straight,2}.

\begin{lemma}\label{straight,2}
Let $\{I_i = [a_i, b_i]\}_{1\le i \le N}\subseteq [0,L]$ be a collection of intervals and  $M\in\mathds{Z^+}$. If for all $x \in [0,L]$, $x \in I_i$ for at most M choices of $i$, then  $\sum\limits_{i=1}^N  \mathcal{L}(I_i)\leq{M}L$. 
\end{lemma}

\begin{proof}
Let $U= \bigcup\limits_{i=1}^{N} \{a_i, b_i\}$ be the set of endpoints. The elements of $U$ partition $[0,L]$ into at most $K\leq{2N+1}$ intervals $J_1, \dots, J_K$. Then
\begin{eqnarray}
\sum\limits_{i=1}^N \mathcal{L}(I_i) &=& \sum\limits_{i=1}^N \sum\limits_{j=1}^K \mathcal{L}(I_i \cap J_j) 
   = \sum\limits_{j=1}^K \sum\limits_{i=1}^N \mathcal{L}(I_i \cap J_j)  \nonumber\\   
   &\leq& \sum\limits_{j=1}^K M\mathcal{L}(J_j)
   = ML\nonumber
\end{eqnarray}
where the inequality follows from the fact that $x \in I_i$ for at most $M$ different choices of $i$. 
\end{proof}

We are now ready to complete the proof of Lemma~\ref{straight}:

\begin{proof}[Proof of Lemma~\ref{straight}] Let $\{u_\lambda\}$ be a minimal partition of $\gamma\cap B$ into intervals such that ${u_\lambda}$ is the graph of a monotone function for each $\lambda$. 

For a given $\lambda$ let $(a_{\lambda_1},b_{\lambda_1})$ and $(a_{\lambda_2},b_{\lambda_2})$ be the endpoints of ${u_\lambda}$ and w.l.o.g assume that $a_{\lambda_2}>a_{\lambda_1}$ and $b_{\lambda_2}>b_{\lambda_1}$ so that $u_\lambda$ is the graph of an increasing function.  Then
\begin{eqnarray}
    \mathcal{L}(u_\lambda) &=& \int_{a_{\lambda_1}}^{a_{\lambda_2}}\sqrt{1 + u_\lambda'(x)^2}\,dx  \nonumber \\
  &\leq& \int_{a_{\lambda_1}}^{a_{\lambda_2}} 1+u_\lambda'(x)\,dx  \nonumber \nonumber\\
  &\leq& |a_{\lambda_2}-a_{\lambda_1}|+|b_{\lambda_2}-b_{\lambda_1}|.\nonumber
\end{eqnarray}

Then by Lemma~\ref{straight,2}
$$
\sum_{\lambda} |a_{\lambda_2}-a_{\lambda_1}| \leq{2M\epsilon} \; \; \; \mbox{and} \; \; \; \sum_\lambda |b_{\lambda_2}-b_{\lambda_1}| \leq{2M\epsilon},
$$
and hence
$$
    \mathcal{L}(\gamma\cap B) = \sum_{\lambda}\mathcal{L}(u_{_\lambda}) \nonumber
    \leq  4 M\epsilon.\nonumber
$$
\end{proof}


\section{Proof of Theorem~\ref{TSC}}\label{sec_locallyfinite}

In this section we prove Theorem~\ref{TSC}. There are two cases depending on whether or not $x\in V$.  The $x\in V$ case is rather simple since the conclusion of Theorem~\ref{TSC} holds at $t=0$ and it is enough to count the number of intersections with static lines.  In the second case, we instead count the number of intersections with the so-called grim reaper, a translating solution to curve shortening flow.

\begin{defn} [Grim reaper]\label{grimreaper} For any $\lambda, c>0$ the translating graph
$$
u^\lambda(x,t)=-\frac{1}{c}\ln(\sec(cx))+3+\lambda-ct,\:\:\:\:\:\: \frac{-\pi}{2c}<x<\frac{\pi}{2c}
$$
is a solution to curve shortening flow.
\end{defn}

As exact solutions are hard to produce the grim reapers have been used extensively in barrier arguments in curve shortening flow.  Bu they also appear naturally.  While the shrinking circle is the only singularity for embedded curves, grim reapers do occur as a blow-up in the immersed case~\cite{Ang2}.

The vertical shift of $3+\lambda$ in Definition~\ref{grimreaper} guarantees that $u^{\lambda}(\cdot,0) \cap V=\emptyset$.  See Figure 2.  The arguments of this section rely on the fact that regardless of how large we choose $c$ to be, i.e. how thin the grim reaper is, the number of intersections with $T$ is bounded.  Let $\Gamma$ be the graph of $\sin(1/x)|_{(0,1]}$.  We then have the following: 


\begin{lemma}\label{intersect}
If $c>1$ and $0\leq\lambda\leq 6$ then $|u^{\lambda}(\cdot,0) \cap\Gamma| = 1$.
\end{lemma}

\begin{proof} Let $0<x_1<x_2$ be such that $(x_1, 1)$ and $(x_2,-1)$ lie on the graph of $u=u^{\lambda}(\cdot,0)$.  Then all intersections occur on the interval $[x_1,x_2]$.  Setting $u=\pm1$ and solving explicitly we obtain
$$
x_1 = \frac{\arccos(e^{(-2-\lambda)c})}{c} \; , \; x_2 = \frac{\arccos(e^{(-4-\lambda)c})}{c}.
$$

Now,
$$\min_{x \in [x_1, x_2]} |u'(x)| = |u'(x_1)| = \tan(\arccos(e^{(-2-\lambda)c})) = \frac{\sqrt{1 - e^{(-4-\lambda)c}}}{e^{(-2-\lambda)c}} > \frac{1}{2} e^{2c},
$$
and
$$\max_{x \in [x_1, x_2]} |\Gamma'|=\max_{x \in [x_1, x_2]}\frac{|\cos(1/{x})|}{{x}^2}\leq \frac{1}{{x_1}^2} = \Big(\frac{c}{\arccos(e^{(-2-\lambda)c})}\Big)^2 < \frac{16}{\pi^2}c^2.
$$
In both cases the final inequality uses the fact $c>1$.  Thus
$$
\min_{x \in [x_1, x_2]}|u'| > \max_{x \in [x_1, x_2]}|\Gamma'|
$$ 
which proves the result.\end{proof}


We now turn to the proof of Theorem~\ref{TSC}.

\begin{proof}[Proof of Theorem~\ref{TSC}] Case 1: Let $x\in\mathds{R}^2\setminus V$, and let $\ell_1\neq \ell_2$ be distinct lines through $x$ such that $\ell_i \cap V = \emptyset$ for $i=1,2$.  Let 
$$
\epsilon = \frac{\mathrm{d}(\ell_1 \cup \ell_2, V)}{2},
$$
and let $\theta=\angle(\ell_1,\ell_2)$. Note that $\epsilon$ and $\theta$ are necessarily small when $x$ is close to $V$. 

For each $i = 1,2$ we define a family of lines as follows: let $L_i$ be the largest family of lines parallel to $\ell_i$ which intersect $B_\epsilon (x)$ non-trivially.  Since the definition of $\epsilon$ implies that each line in $L_i$ is disjoint from $V$ there exists a constant $C>0$ such that each line intersects $T$ at most $C$ times. 

Lemma~\ref{intbound} implies that $|\gamma_n\cap\zeta|\leq 2C$ for each line $\zeta\in L_i$.  Hence for each $t>0$ we have $|(\gamma_n)_t\cap\zeta|\leq 2C$ since $\zeta$ is a static solution to curve shortening flow and the number of intersections is non-increasing in time.

By translating  and rotating we may assume that $x=0$ and that the lines in $L_1$ are parallel to the $x$-axis.  Then the linear map 
$$
A:(x,y)\mapsto (x\tan\theta - y,y)
$$
takes each line in $L_2$ to a vertical line and sends each horizontal line onto itself.  A simple calculation yields
$$
\|A\|_2,\|A^{-1}\|_2\leq 2\max\{\tan\theta,(\tan\theta)^{-1}\}=:d(\theta)
$$
and hence by Theorem~\ref{jacobian} we conclude that
$$
\mathcal{L}((\gamma_n)_t\cap B_\epsilon(x))<4Cd(\theta)^2\epsilon,
$$
which proves the result.  Notice that in this case the constant on the right-hand side is independent of time.


Case 2: Now suppose $x\in V$.   Lemma~\ref{intersect} and the choice of $\gamma_n$ implies that there exists a constant $C>0$ such that for each $0\leq\lambda\leq 6$ we have 
$$
|\gamma_n\cap u^\lambda|\leq C.
$$
For the topologist's sine curve in Figure 1 we have $C=4$, but we allow for the possibility that the arc added to join the $\sin(1/x)$ portion back to the origin introduces a higher number of intersections.

Given $t_0 > 0$ define $c  = \frac{6}{t_0}$ so that the grim reaper defined by 
$$
u(x,t) = -\frac{1}{c}\ln(\sec(cx)) + 3 - ct 
$$
passes completely through $V$ by $t_0$.  

Now fix $\alpha\ll 1$ and choose $\epsilon > 0$ such that $|u'(x)| < \alpha$ for all $x \in [-\epsilon, \epsilon]$. Restricting the family of curves~$u^\lambda_{t_0}$ to $B_\epsilon(x)$ we obtain a foliation by vertical translates. 
Using Lemma~\ref{intersect} and the fact that the number of intersections does increase under curve shortening flow we have for each $0\leq\lambda\leq 6$
$$
|(\gamma_n)_{t_0}\cap u^\lambda_{t_0}|\leq C.
$$
Moreover, by Definition~\ref{goodapprox}, $\gamma_n$, and hence $(\gamma_n)_{t_0}$, intersect each vertical line at most $\omega$ times. 

The map $f:(x,y)\mapsto(x,y-u(x))$ sends each vertical line onto itself, each curve $u^\lambda_{t_0}$ to a horizontal line, and since $u'(x)$ is small $f$ satisfies the hypothesis of Theorem~\ref{jacobian} with $d=2$.  Thus 
$$
\mathcal{L}((\gamma_n)_{t_0}\cap B_\epsilon(x))<4\max\{C,\omega\}\epsilon.
$$
\end{proof}



\section{Finite $\mathcal{H}^1$-measure at positive times}\label{sec_locallyconnected}

We begin this section by reviewing the definition of the Hausdorff 1-measure.

\begin{defn} [$\mathcal{H}^1$] Let $K$ be compact.  For any $\delta>0$ we first define
$$
\mathcal{H}^1_\delta(K) =\inf\left\{\sum_n\mathrm{diam}(U_n)\mid U_n\:\mathrm{open}, K\subset\bigcup_n U_n,\:\: \mathrm{diam}(U_n)<\delta\right\}.
$$
As $\delta\to 0$ the quantity above is monotonic and hence the limit exists, although it may be infinite.
$$
\mathcal{H}^1(K)=\lim_{\delta\to 0^+}\mathcal{H}^1_\delta(K).
$$
\end{defn}

In general the convergence properties of the Hausdorff measure under Hausdorff convergence are quite poor.  Indeed, any compact set is the limit of finite sets of points and hence the $\mathcal{H}^1$-measure may jump in the limit.  Nevertheless, in the present situation we do obtain finite $\mathcal{H}^1$-measure of the limit. 

\begin{lemma} Let $\Omega\subset\mathds{R}^2$ be a bounded simply-connected domain and suppose that $\{E_i\}$ is a sequence of smooth disks satisfying $E_1\subset E_2\subset E_3\subset\ldots$ and $\cup E_n=
\Omega$.  If $$\sup_n\{\mathcal{H}^1(\partial E_i)\}<\infty$$ then $\mathcal{H}^1(\partial \Omega)<\infty$.
\end{lemma}

\begin{proof} Let $\gamma_n=\partial E_n$, $L=\sup_n\{\mathcal{H}^1(\gamma_n)\}$ and fix $n\in\mathds{N}$.  Given $\epsilon>0$, let $x_1,x_2,\ldots$ be an ordered sequence of points along $\gamma_n$ such that 
\begin{enumerate}
\item $d(x_i,x_{i+1})=\frac{\epsilon}{2}$, and
\item the open arc connecting $x_i$ to $x_{i+1}$ is contained in $B_{\frac{\epsilon}{2}}(x_i)$.
\end{enumerate}

The first condition implies that the number of points $x_i$ is at most $\left\lceil{\frac{2L}{\epsilon}}\right\rceil$ and the second guarantees that 
$$
B_{\frac{\epsilon}{2}}(\gamma_n)\subset\bigcup_i B_{\epsilon}(x_i).
$$
Now, if $n$ is sufficiently large then $\partial\Omega\subset B_{\frac{\epsilon}{2}}(\gamma_n)$.  Thus
$$
\mathcal{H}^1_\epsilon(\partial\Omega)\leq 2\epsilon\left\lceil{\frac{2L}{\epsilon}}\right\rceil<4L+1
$$
and the result follows by taking $\epsilon\to 0$.
\end{proof}

\section{Proof of Theorem~\ref{main}}\label{sec_proof}

We now complete the proof of Theorem~\ref{main}:

\begin{proof} Let $\gamma_n$ be a sequence of approximations satisfying Definition~\ref{goodapprox} and let $K_t$ be the Hausdorff limit of the sequence $(\gamma_n)_t$.  Then $K_t\subset\partial(T_t)$.  Theorem~\ref{TSC} and Theorem~\ref{measure} imply that $\mathcal{H}^1(K_t)<\infty$ and hence that $K_t$ is locally-connected~\cite{F}.  Theorem~\ref{Lauer} then implies that $K_t$ is a smooth curve shortening flow for $t>0$.

If one chooses instead an approximating sequence in $\mathds{R}^2\setminus\overline{\Omega}$ then the same conclusion holds. Thus for positive times $\partial(T_t)$ is the union of two smooth curves, but Lemma~\ref{measure} implies that $T_t$ has measure zero and hence that the two curves coincide.  This completes the proof.
\end{proof}


Department of Mathematics, M.I.T., Cambridge, MA, USA, 02139\\
\phantom{ttt} Email address: caseylam@mit.edu

\bigskip

Fachbereich Mathematik und Informatik, Freie Universit$\ddot{a}$t, Berlin, Germany, 14195\\
\phantom{ttt} Email address: lauer@zedat.fu-berlin.de

\end{document}